\documentclass[11pt]{amsart}
\usepackage[top=1.5in, bottom=1.5in, left=1.5in, right=1.5in]{geometry}
\usepackage{amsmath}
\usepackage{amssymb}
\usepackage{hyperref}
\usepackage{amsthm}
\usepackage{mathtools}

\newtheorem{thm}{Theorem}
\newtheorem{conj}[thm]{Conjecture}

\newtheorem{Remark}[thm]{Remark}
\newenvironment{remark}
  {\begin{Remark}\rm}{\end{Remark}}

\newcommand{\emailhref}[1]{\email{\href{#1}{#1}}}

\title[Two $t$-analogues of the tree inversion enumerator]{Two $t$-analogues of the\\ tree inversion enumerator}
\author{Sam Hopkins}\emailhref{samuelfhopkins@gmail.com}
\address{Department of Mathematics, Howard University, Washington, DC}
\date{May 21, 2026}

\begin{document}

\begin{abstract}
In this note, we introduce two $t$-analogues $I_n(q,t)$ and $\widetilde{I}_n(q,t)$ of the tree inversion enumerator $I_n(q)$. Although similar, $I_n(q,t)$ and $\widetilde{I}_n(q,t)$ are different. But they both seem to have interesting properties. In particular, we conjecture that their $q=-1$ specializations give two different, natural refinements of the zigzag numbers counting alternating permutations.
\end{abstract}

\maketitle

For $n \geq 0$, let $\mathrm{Tree}(n+1)$ denote the set of trees on vertex set $\{0,1,\ldots,n\}$. Cayley's formula says $\#\mathrm{Tree}(n+1)=(n+1)^{n-1}$. We think of a tree $T \in \mathrm{Tree}(n+1)$ as rooted at the vertex~$0$, and define an \emph{inversion} of such a $T$ to be a pair $(i,j)$ with $1 \leq i < j \leq n$ such that $j$ is on the unique path from $i$ to the root $0$. We use $\mathrm{inv}(T)$ to denote the number of inversions of $T$. We then define
\begin{equation} \label{eq:inv_enum_tree}
    I_n(q) := \sum_{T\in \mathrm{Tree}(n+1)}q^{\mathrm{inv}(T)}
\end{equation}
to be the tree inversion enumerator. A lot is known about $I_n(q)$. For example, Kreweras~\cite{kreweras1980famille} showed that $I_n(q)$ is determined by the recurrence
\begin{align} \label{eqn:recur}
\notag I_0(q) &= 1;\\
 I_n(q) &= \sum_{i=0}^{n-1}\binom{n-1}{i}(1+q+\cdots+q^{i})I_i(q)I_{n-1-i}(q) \textrm{ for $n\geq 1$}.
\end{align}

For $n \geq 0$, a \emph{parking function} of length $n$ is a sequence $\pi=(\pi_1,\ldots,\pi_n)$ of~$n$ positive integers whose weakly increasing rearrangement $\pi_{i_1}\leq\pi_{i_2}\leq\cdots\leq\pi_{i_n}$ satisfies $\pi_{i_j} \leq j$ for all $1\leq j \leq n$. Let $\mathrm{PF}(n)$ denote the set of parking functions of length~$n$. It is well-known that also $\#\mathrm{PF}(n)=(n+1)^{n-1}$. In fact, defining the \emph{cosum} of a parking function $\pi\in\mathrm{PF}(n)$ to be $\mathrm{cosum}(a):=\binom{n+1}{2}-\sum_{i=1}^{n}\pi_i$, Kreweras~\cite{kreweras1980famille} showed that
\begin{equation} \label{eq:inv_enum_pf}
    I_n(q)=\sum_{\pi\in\mathrm{PF}(n)}q^{\mathrm{cosum}(\pi)}
\end{equation}
as well. He did this by showing that the parking function cosum enumerator satisfies the same recurrence~\eqref{eqn:recur} as the tree inversion enumerator.

Here we introduce two $t$-analogues of $I_n(q)$, one based on~\eqref{eq:inv_enum_tree} and one based on~\eqref{eq:inv_enum_pf}. For $T\in\mathrm{Tree}(n+1)$, let $\mathrm{lev}(T)$ denote the number of \emph{leaves} of~$T$. And for $\pi \in \mathrm{PF}(n)$, let $\mathrm{exced}(\pi)$ denote the number of \emph{(strict) excedances} of~$\pi$, i.e., the number of $1\leq i \leq n$ such that $\pi_i > i$. Then we define
\begin{align*}
I_n(q,t) &:= \sum_{T\in \mathrm{Tree}(n+1)}q^{\mathrm{inv}(T)}t^{\mathrm{lev(T)-1}};\\
\widetilde{I}_n(q,t) &:= \sum_{\pi\in \mathrm{PF}(n)} q^{\mathrm{cosum}(\pi)}t^{\mathrm{exced}(\pi)}.
\end{align*}

There are many similarities between $I_n(q,t)$ and $\widetilde{I}_n(q,t)$. For instance, they both interpolate between the Mahonian and Eulerian distributions on the symmetric group $S_n$. For a permutation $\sigma \in S_n$, we use $\mathrm{inv}(\sigma)$ to denote the number of inversions of $\sigma$ and $\mathrm{des}(\sigma)$ to denote the number of descents of $\sigma$. We then recall that $\sum_{\sigma \in S_n}q^{\mathrm{inv}(\sigma)} = [n]_q!$ is the \emph{Mahonian polynomial} (where we use the standard $q$-number notation $[k]_q=(1+q+\cdots+q^{k-1})$ and $[n]_q!=[n]_q[n-1]_q\cdots[1]_q$), while $\sum_{\sigma \in S_n}t^{\mathrm{des}(\sigma)} = \sum_{k=0}^{n}A(n,k)t^k=A_n(t)$ is the \emph{Eulerian polynomial}.

\begin{thm}
We have
\[I_n(q,0)=\widetilde{I}_n(q,0)= [n]_q!\]
and
\[I_n(0,t)=\widetilde{I}_n(0,t) = A_n(t).\]
\end{thm}
\begin{proof}
To show $I_n(q,0)=[n]_q!$: this is clear because trees with one leaf are permutations in a natural way, and the tree inversions are inversions of the permutation.

To show $\widetilde{I}_n(q,0) = [n]_q!$: parking functions $\pi\in \mathrm{PF}(n)$ with $\mathrm{exced}(\pi)=0$ are sequences of integers $(\pi_1,\pi_2,\ldots,\pi_n)$ with $1 \leq \pi_i \leq i$ for all $i$. (These are essentially Lehmer codes of permutations.) The cosum enumerator of these is evidently $[n]_q!$.

To show $I_n(0,t)=A_n(t)$: it is well-known that trees $T \in \mathrm{Tree}(n+1)$ with $\mathrm{inv}(T)=0$, i.e., \emph{increasing trees}, are also in bijection with permutations in $S_n$. Indeed, given such an increasing tree, perform a depth-first search on the tree, starting at the root $0$, and always preferring to visit the vertex with the biggest label when one has a choice. Record the order in which vertices are visited in this search to obtain a permutation. This is a bijective procedure. Moreover, the leaves of the tree correspond to the descents of the permutation, except for the ``last'' leaf.

To show $\widetilde{I}_n(0,t)=A_n(t)$: parking functions with cosum zero are permutations, and it is well-known that the generating function of permutations by excedances is the Eulerian polynomial.
\end{proof}

\begin{table}
\begin{center}
\renewcommand{\arraystretch}{1.25}
\begin{tabular}{ c|c|c } 
$n$ & $I_n(q,t)$ & $\widetilde{I}_n(q,t)$ \\ \hline
$1$ & $1$ & $1$ \\ \hline
$2$ & $q+t+1$ & $q+t+1$ \\ \hline
$3$ & $q^3 \! + q^2t \! + 2q^2 \! + 4qt \! + t^2 \! + 2q \! + 4t \! + 1$ & $q^3 \! + q^2t \! + 2q^2 \! + 4qt \! + t^2 \! + 2q \! + 4t \! + 1$ \\ \hline
$4$ & \parbox{2.5in}{\begin{center}$q^6 + q^5t + 3q^5 + 5q^4t + q^3t^2 + 5q^4 + 13q^3t + 5q^2t^2 + 6q^3 + 20q^2t + \mathbf{11}qt^2 + t^3 + 5q^2 + \mathbf{22}qt + 11t^2 + 3q + 11t + 1
$\end{center}} & \parbox{2.5in}{\begin{center}$q^6 + q^5t + 3q^5 + 5q^4t + q^3t^2 + 5q^4 + 13q^3t + 5q^2t^2 + 6q^3 + 20q^2t + \mathbf{12}qt^2 + t^3 + 5q^2 + \mathbf{21}qt + 11t^2 + 3q + 11t + 1$\end{center}} \\
\end{tabular}
\end{center}
\caption{The two $t$-analogues $I_n(q,t)$ and $\widetilde{I}_n(q,t)$ of $I_n(q)$ for small values of $n$. Notice that they differ for $n=4$.} \label{tab:qts}
\end{table}

But in spite of these similarities, $I_n(q,t)$ and $\widetilde{I}_n(q,t)$ are not the same in general. Table~\ref{tab:qts} displays $I_n(q,t)$ and $\widetilde{I}_n(q,t)$ for $n=1,\ldots,4$. Notice that $I_4(q,t)\neq \widetilde{I}_4(q,t)$. In fact, even their $q=1$ specializations are different:
\[I_4(1,t) = t^3 + 28t^2 + 72t + 24 \neq t^3 + 29t^2 + 71t + 24=\widetilde{I}_4(q,t).\] 

In order to compare $I_n(q,t)$ and $\widetilde{I}_n(q,t)$, it is helpful to view both of them as generating functions on the same set of objects. Actually, $I_n(q,t)$ was already studied by Stanley and Yin~\cite{stanley2023enumerative}, and they were able to express it as a generating function of natural statistics on parking functions. (More precisely, Stanley and Yin studied two \emph{four variable} generating functions, one over labeled trees and one over parking functions, and showed that they are equal. Our $I_n(q,t)$ is a two variable specialization of their four variable generating function.)

To explain what Stanley and Yin showed about $I_n(q,t)$, we need to recall the parking procedure for parking functions. Imagine $n$ cars, in sequence, drive down a one-way street, which has parking spots labeled $1$ to $n$. The~$i$th car prefers spot $\pi_i$ and will park there if possible, but if that spot is occupied they will continue driving and park at the next unoccupied spot, if there is one. The sequence $\pi=(\pi_1,\ldots,\pi_n)$ of parking spot preferences results in all cars successfully parking if and only if $\pi$ is a parking function. For $\pi \in \mathrm{PF}(\pi)$ we use $\mathrm{oc}(\pi)$ to denote the \emph{parking outcome permutation}, i.e., $\mathrm{oc}(\pi)=\sigma_1\sigma_2\ldots \sigma_n$ means that car $i$ ends up at spot $\sigma_1$. Also note that if $\mathrm{oc}(\pi)^{-1}=\tau_1\ldots \tau_n$ this means that car $\tau_i$ ends up in spot~$i$. For example, if~$\pi = (2, 2, 1, 3)$ then $\mathrm{oc}(\pi) = 2314$ and $\mathrm{oc}(\pi)^{-1}
= 3124$.

\begin{thm}[Stanley--Yin~\cite{stanley2023enumerative}]
$I_n(q,t) = \sum_{\pi \in \mathrm{PF}(n)}q^{\mathrm{cosum}(a)}t^{\mathrm{des}(\mathrm{oc}(\pi)^{-1})}$.
\end{thm}

On the other hand, $\widetilde{I}_n(q,t)$ is not a specialization of the generating function studied by Stanley and Yin. Nevertheless, computational evidence strongly suggests the following, which would make $I_n(q,t)$ and $\widetilde{I}_n(q,t)$ appear very similar indeed.

\begin{conj} \label{conj:des}
$\widetilde{I}_n(q,t) = \sum_{\pi \in \mathrm{PF}(n)}q^{\mathrm{cosum}(a)}t^{\mathrm{des}(\mathrm{oc}(\pi))}$.
\end{conj}

Note that for $\pi \in \mathrm{PF}(n)$, $\mathrm{exced}(\pi)$ and $\mathrm{des}(\mathrm{oc}(\pi))$ will in general be different. For example for $\pi=(4,2,1,3)$, $\mathrm{exed}(\pi)=1$ but $\mathrm{des}(\mathrm{oc}(\pi))=\mathrm{des}(4213)=2$. So to prove Conjecture~\ref{conj:des}, one may need to define a nontrivial bijection on $\mathrm{PF}(n)$.

\begin{remark}
Stanley and Yin~\cite{stanley2023enumerative} showed that $I_n(q,t)$ satisfies the following $t$-deformation of the recurrence~\eqref{eqn:recur}:
\[I_n(q,t) \hspace{-0.05cm} = \hspace{-0.05cm} (1+q+\cdots+q^{n-1})I_{n-1}(q,t) + t\cdot \sum_{i=0}^{n-2}\binom{n-1}{i}(1+q+\cdots+q^i)I_i(q,t)I_{n-1-i}(q,t).\]
We have not been able to guess a recurrence like this for $\widetilde{I}_n(q,t)$. Relatedly, we have not been able to guess how to express $\widetilde{I}_n(q,t)$ as a generating function over trees.
\end{remark}

Recall a permutation $\sigma=\sigma_1\cdots\sigma_n\in S_n$ is \emph{alternating} if $\sigma_1 < \sigma_2 > \sigma_3 < \cdots$. The number of alternating permutations in $S_n$ is given by the \emph{zigzag number} (a.k.a.~``\emph{Euler number}'')~$E_n$. The sequence of $E_n$, for $n \geq 0$, begins $1, 1, 1, 2, 5, 16, 61, 272,  \ldots$. One remarkable property of the tree inversion enumerator is that $I_n(-1) = E_n$. (This can be proved directly from the recurrence~\eqref{eqn:recur}; see also, e.g.,~\cite{kuznetsov1994increasing}.)

Our most interesting conjectures about $I_n(q,t)$ and $\widetilde{I}_n(q,t)$ concern this $q=-1$ specialization. Table~\ref{tab:minusone} displays $I_n(-1,t)$ and $\widetilde{I}_n(-1,t)$ for small $n$. Experimentally, it appears that we obtain two different, natural refinements of the zigzag numbers~$E_n$: what we might call the ``\emph{Andr\'{e}}'' or ``\emph{simsun Eulerian numbers}'' $B(n,k)$, and the ``\emph{zigzag Eulerian numbers}'' $C(n,k)$.

\begin{table}
\begin{center}
\renewcommand{\arraystretch}{1.25}
\begin{tabular}{ c|c|c } 
$n$ & $I_n(-1,t)$ & $\widetilde{I}_n(-1,t)$ \\ \hline
$1$ & $1$ & $1$ \\ \hline
$2$ & $t$ & $t$ \\ \hline
$3$ & $t^2+t$ & $t^2+t$ \\ \hline
$4$ & $t^3 + 4t^2$ & $t^3 + 3t^2 + t$ \\ \hline
$5$ & $t^4 + 11t^3 + 4t^2$ & $t^4 + 7t^3 + 7t^2 + t$ \\ \hline
$6$ & $t^5 + 26t^4 + 34t^3$ & $t^5 + 14t^4 + 31t^3 + 14t^2 + t$ \\ \hline
$7$ & $t^6 + 57t^5 + 180t^4 + 34t^3$ & $t^6 + 26t^5 + 109t^4 + 109t^3 + 26t^2 + t$
\end{tabular}
\end{center}
\caption{The specializations of $I_n(q,t)$ and $\widetilde{I}_n(q,t)$ at $q=-1$.} \label{tab:minusone}
\end{table}

\begin{conj} For $n\geq 1$, write
\[I_n(-1,t)=\sum_{k=0}^{\lfloor (n-1)/2\rfloor}B(n,k)t^{n-1-k} \textrm{ and } \widetilde{I}_n(-1,t)=\sum_{k=0}^{n-2}C(n,k)t^{k+1}.\]
Then:
\begin{enumerate}
\item $B(n,k)$ is the number of ``\emph{simsun permutations}'' in $S_{n-1}$ with $k$ descents (see, e.g.,~\cite{chow2011counting} and \url{https://oeis.org/A094503});
\item $C(n,k)$ is the number of alternating permutations in $S_n$ with $k$ ``\emph{big returns}'' (see, e.g.,~\cite{petersen2025zigzag} and \url{https://oeis.org/A205497}).
\end{enumerate}
\end{conj}

\subsection*{Acknowledgments} 
I would like to thank the organizers and participants in the Special Session on Parking Functions and their Generalizations at the 2025 Spring Eastern Sectional Meeting of the AMS in St.~Louis, MO. This conference inspired me to think about trees and parking functions again after a long hiatus.

\begin{remark}
After the first version of this note was posted online, the conjectures here were all resolved, affirmatively, in~\cite{defant2026short}. The proofs there were obtained autonomously by ChatGPT 5.4 Pro.
\end{remark}

\bibliography{main}{}
\bibliographystyle{abbrv}

\end{document}